\documentclass{commat}

\renewcommand{\geq}{\geqslant}
\renewcommand{\leq}{\leqslant}
\renewcommand{\phi}{\varphi}
\renewcommand{\vec}[1]{\mathbf{#1}}

\newcommand{\cC}{\mathcal{C}}
\newcommand{\cL}{\mathcal{L}}
\newcommand{\cP}{\mathcal{P}}
\newcommand{\e}{\varepsilon}

\newcommand{\La}{\Lambda}
\newcommand{\Q}{\mathbb{Q}}
\newcommand{\R}{\mathbb{R}}
\newcommand{\Z}{\mathbb{Z}}

\title{%
    On triviality of uniform Diophantine exponents of lattices
    }

\author{%
    Oleg German
    }

\authorinfo{National Research University Higher School of Economics, Russia}{german.oleg@gmail.com}

\abstract{%
    In this paper we prove that uniform Diophantine exponents of lattices attain only trivial values.
    }



\VOLUME{31}
\NUMBER{2}
\firstpage{27}
\DOI{https://doi.org/10.46298/cm.11137}

\begin{document}

\section{Introduction}\label{sec:intro}

Most of the problems in the theory of homogeneous linear Diophantine approximation can be formulated within the following general setting. Given $n$ linearly independent linear forms $L_1,\ldots,L_n$ in $d$ real variables, $n\leq d$, one of the basic questions is how small the vector $(L_1(\vec u),\ldots,L_n(\vec u))$ can be as $\vec u$ ranges through nonzero integer points. There are several standard ways to measure the ``size'' of this vector. One way is to consider an arbitrary norm, say, the supremum norm. This is the case when dealing with the most classical problems of Diophantine approximation. Another way is to alter the supremum norm turning it into a~so called weighted norm, which leads to weighted Diophantine approximation. One can also measure a~vector by considering the product of the absolute values of its entries. This approach leads to the area of multiplicative Diophantine approximation, which is full of most intriguing problems. It suffices to mention that the famous Littlewood conjecture belongs to this area.

Whatever functional we choose to measure the ``size'' of a~vector, the notion of a~Diophantine exponent emerges naturally as the simplest quantitative characteristic responsible for the approximation properties of the $n$-tuple $(L_1,\ldots,L_n)$.

The cases $n<d$ and $n=d$ have to be distinguished. The former case, along with ordinary Diophantine exponents, provides multiplicative Diophantine exponents, whereas the latter one provides Diophantine exponents of lattices.

The main purpose of the current note is to study the possible values of uniform Diophantine exponents of lattices.

The rest of the paper is organized as follows. In Section~\ref{sec:lattice exponents} we give the definitions of regular and uniform Diophantine exponents of lattices, discuss what is known about their spectra, and formulate the main result of this paper. In Section~\ref{sec:proof} we formulate Davenport's lemma and apply it to prove our main theorem. Finally, in Section~\ref{sec:weak_uniform_exponents} we address the ``weak'' version of uniform lattice exponents and their analogues in the case $n<d$.

\section{Diophantine exponents of lattices}\label{sec:lattice exponents}

Given $\vec z=(z_1,\ldots,z_d)\in\R^d$, we denote
\[
  |\vec z|=\max_{1\leq i\leq d}|z_i|,
  \qquad
  \Pi(\vec z)=\prod_{
\begin{subarray}
{c}1\leq i\leq d
\end{subarray}
}|z_i|^{1/d}.
\]
For each $d$-tuple $\pmb\lambda=(\lambda_1,\ldots,\lambda_d)\in\R_+^d$ define the parallelepiped $\cP(\pmb\lambda)$ as
\begin{equation}\label{eq:prallelepipeds_lattice_exp}
  \cP(\pmb\lambda)=\Big\{\,\vec z=(z_1,\ldots,z_d)\in\R^d \ \Big|\ |z_i|\leq\lambda_i,\ i=1,\ldots,d \Big\}.
\end{equation}

Let $L_1,\ldots,L_d$ be linearly independent linear forms in $d$ real variables.
Consider the lattice
\[
  \La=\Big\{ (L_1(\vec u),\ldots,L_d(\vec u)) \,\Big|\ \vec u\in\Z^d \Big\}.
\]

\begin{definition}\label{def:regular_lattice_exponents}
  The supremum of real numbers $\gamma$ for which there exists $t$ however large and a~$d$-tuple $\pmb\lambda\in\R_+^d$ satisfying
  \[
    |\pmb\lambda|=t,
    \qquad
    \Pi(\pmb\lambda)=t^{-\gamma},
  \]
  such that $\cP(\pmb\lambda)$ contains a~nonzero point of $\La$, is called the \emph{(regular) Diophantine exponent} of $\La$ and is denoted by $\omega(\La)$.
\end{definition}

\begin{definition}\label{def:uniform_lattice_exponents}
  The supremum of real numbers $\gamma$ such that for every $t$ large enough and every $d$-tuple $\pmb\lambda\in\R_+^d$ satisfying
  \[
    |\pmb\lambda|=t,
    \qquad
    \Pi(\pmb\lambda)=t^{-\gamma},
  \]
  the parallelepiped $\cP(\pmb\lambda)$ contains a~nonzero point of $\La$, is called the \emph{uniform Diophantine exponent} of $\La$ and is denoted by $\hat\omega(\La)$.
\end{definition}

It is a~straightforward consequence of Minkowski's convex body theorem that
\[
  \omega(\La)\geq\hat\omega(\La)\geq0.
\]

It is natural to expect that $\omega(\La)$ can attain any non-negative value. For $d=2$ this statement is easily deduced from the theory of continued fractions. However, for $d\geq3$, it is still unproved. The best bound $\omega_d$ known up to now such that for every $\omega\geqslant\omega_d$ there is a~lattice $\La$ such that $\omega(\La)=\omega$ is equal to $3-d(d-1)^{-2}$, see~\cite{german_lattice_exponents_spectrum}.

As for the possible values of the uniform exponent $\hat\omega(\La)$, nothing has been known so far except for the two obviously attainable values -- zero and infinity. Indeed, $\hat\omega(\La)$ is obviously equal to $\infty$ whenever $\La$ is a~sublattice of $\Z^d$. As for the zero value, it is attained, for instance, on a~lattice
\begin{equation*}
  \La=
  \begin{pmatrix}
    \sigma_1(\theta_1) & \sigma_1(\theta_2) & \cdots & \sigma_1(\theta_d) \\
    \sigma_2(\theta_1) & \sigma_2(\theta_2) & \cdots & \sigma_2(\theta_d) \\
    \vdots & \vdots & \ddots & \vdots \\
    \sigma_d(\theta_1) & \sigma_d(\theta_2) & \cdots & \sigma_d(\theta_d)
  \end{pmatrix}
\Z^d,
\end{equation*}
where $\theta_1,\ldots,\theta_d$ is a~basis of a~totally real algebraic extension of $\Q$ of degree $d$, and $\sigma_1,\ldots,\sigma_d$ are the embeddings of this extension into $\R$. Such lattices are called algebraic. More details on algebraic lattices can be found in~\cite{borevich_shafarevich}. Due to Dirichlet's unit theorem (see also~\cite{borevich_shafarevich}), the functional $\Pi(\cdot)$ is bounded away from zero at nonzero points of an algebraic lattice, hence both the regular exponent of such a~lattice, and the uniform one have to be equal to zero.

The following statement is the main result of the paper.

\begin{theorem}\label{t:the_spectrum_of_the_uniform_lattice_exponent}
  Let $\La$ be a~full-rank lattice in $\R^d$. If $\La$ is similar to a~sublattice of $\Z^d$ modulo the action of the group of non-degenerate diagonal operators, then $\hat\omega(\La)=\infty$. Otherwise, $\hat\omega(\La)=0$.
\end{theorem}

A lattice $\La$ is the image of a~sublattice of $\Z^d$ under the action of a~non-degenerate diagonal operator if and only if every coordinate axis contains nonzero points of $\La$. Thus, Theorem~\ref{t:the_spectrum_of_the_uniform_lattice_exponent} follows immediately from Theorem~\ref{t:triviality_of_the_uniform_lattice_exponent}.

\begin{theorem}\label{t:triviality_of_the_uniform_lattice_exponent}
  Let $\La$ be a~full-rank lattice in $\R^d$. Let $\ell$ be the first coordinate axis in $\R^d$, i.e.
  \[
    \ell=\Big\{ \vec z=(z_1,\ldots,z_d)\in\R^d \ \Big|\ z_i=0,\ i=2,\ldots,d \Big\}.
  \]
  Suppose $\ell$ contains no points of $\La$ other than the origin. Then $\hat\omega(\La)=0$.
\end{theorem}

\section{Davenport's lemma and the proof of Theorem~\ref{t:triviality_of_the_uniform_lattice_exponent}}\label{sec:proof}

The following statement was essentially proved by Davenport in~\cite{davenport_AA_1936}. A detailed proof can be found in Schmidt's book~\cite{schmidt_DA}.

\begin{theorem}[Davenport's lemma]\label{t:davenport_1937}
  Let $\La$ be a~full-rank 
  lattice in $\R^d$. Consider an arbitrary $d$-tuple $\pmb\lambda=(\lambda_1,\ldots,\lambda_d)\in\R_+^d$ such that
  \[
    \prod_{i=1}^{d}\lambda_i=1.
  \]
  For each $k=1,\ldots,d$, let $\mu_k=\mu_k\big(\cP(\pmb\lambda),\La\big)$ denote the $k$-th successive minimum of $\cP(\pmb\lambda)$ w.r.t. $\La$. Then there is a~permutation $k_1,\ldots,k_d$ of the indices $1,\ldots,d$ and a~positive constant $c$ depending only on $d$ such that there are no nonzero points of $\La$ in $\cP(\pmb\lambda')$, where $\pmb\lambda'=(\lambda'_1,\ldots,\lambda'_d)$,
  \begin{equation}\label{eq:davenport_1937_lambda_prime}
    \lambda'_i=c\mu_{k_i}\lambda_i,
    \qquad
    i=1,\ldots,d.
  \end{equation}
\end{theorem}

\begin{proof}[Proof of Theorem~\ref{t:triviality_of_the_uniform_lattice_exponent}]
  Let $p$ be the smallest integer such that there is a~$p$-dimensional subspace $\cL$ of $\R^d$ with the following two properties:

  1) $\cL$ is spanned by some vectors of $\La$;

  2) $\ell\subset\cL$.
  \\
  Note that these properties determine $\cL$ uniquely.

  \textbf{Case 1.}
  Suppose $p=d$, i.e. there is no proper subspace of $\R^d$ spanned by vectors of $\La$ which contains $\ell$. Fix an arbitrarily small $\e>0$. Then the cylinder
  \begin{equation}\label{eq:cylinder}
    \cC_\e=\Big\{ \vec z=(z_1,\ldots,z_d)\in\R^d \ \Big|\ |z_i|<\e,\ i=2,\ldots,d \Big\}
  \end{equation}
  contains $d$ linearly independent points of $\La$. Hence, for every such an $\e$, there is a~positive $\lambda$ and a~$d$-tuple $\pmb\lambda=(\lambda_1,\ldots,\lambda_d)\in\R_+^d$ such that
  \[
    \lambda_1=\lambda,
    \qquad
    \lambda_i=\lambda^{-1/(d-1)},
    \qquad
    i=2,\ldots,d,
  \]
  and
  \[
    \mu_d\big(\cP(\pmb\lambda),\La\big)\lambda^{-1/(d-1)}<\e.
  \]
  Consider the parallelepiped $\cP(\pmb\lambda')$ with $\pmb\lambda'$ defined by~\eqref{eq:davenport_1937_lambda_prime}. By Theorem~\ref{t:davenport_1937} $\cP(\pmb\lambda')$ contains no nonzero points of $\La$. On the other hand, $\cP(\pmb\lambda')$ is contained in $\mu_d\big(\cP(\pmb\lambda),\La\big)\cP(\pmb\lambda)$, i.e. it is a~subset of $\cC_\e$. Taking into account that the volume of $\cP(\pmb\lambda')$ is bounded away from zero and that $\e$ is arbitrarily small, we get $\hat\omega(\La)=0$.

  \textbf{Case 2.}
  Suppose $p<d$, i.e. $\cL$ is a~proper subspace of $\R^d$. Without loss of generality, we may assume that among the Pl\"ucker coordinates of $\cL$ the greatest one refers to the first $p$ coordinates in $\R^d$. Denote by $\cL'$ the subspace of the first $p$ coordinates in $\R^d$, i.e.
  \[
    \cL'=\Big\{ \vec z=(z_1,\ldots,z_d)\in\R^d \ \Big|\ z_i=0,\ i=p+1,\ldots,d \Big\}.
  \]
  Denote also by $\La'$ the orthogonal projection of $\La\cap\cL$ onto $\cL'$. Then $\La'$ is a~lattice of rank $p$ with determinant depending only on $\La$ and $\cL$. Take an arbitrarily small $\e>0$. The cylinder $\cC_\e$ defined by~\eqref{eq:cylinder} contains $p$ linearly independent points of $\La'$. Hence, same as in Case 1, for every such an $\e$, there is a~positive $\lambda$ and a~$d$-tuple $\pmb\lambda=(\lambda_1,\ldots,\lambda_d)\in\R_+^d$ such that
  \[
    \lambda_1=\lambda,
    \qquad
    \lambda_i=\lambda^{-1/(p-1)},
    \qquad
    i=2,\ldots,p,
    \qquad
    \lambda_j=0,
    \qquad
    j=p+1,\ldots,d.
  \]
  and
  \[
    \mu_p\big(\cP(\pmb\lambda),\La'\big)\lambda^{-1/(p-1)}<\e.
  \]
  Notice that $\cP(\pmb\lambda)$ is a~$p$-dimensional parallelepiped contained in $\cL'$. Therefore, if we set $\mu_k=\mu_k\big(\cP(\pmb\lambda),\La'\big)$, $k=1,\ldots,p$, and apply Minkowski's second theorem, we get
  \begin{equation}\label{eq:prod_of_the_mus}
    \prod_{k=1}^{p}\mu_k
    \asymp
    \det\La'
  \end{equation}
  with the implied constants depending only on $p$.

  By Theorem~\ref{t:davenport_1937} there is a~permutation $k_1,\ldots,k_p$ of the indices $1,\ldots,p$ and a~positive constant $c'$ depending only on $p$ such that there are no nonzero points of $\La'$ in $\cP(\pmb\lambda')$, where $\pmb\lambda'=(\lambda'_1,\ldots,\lambda'_d)$,
  \begin{equation*}
    \lambda'_i=c'\mu_{k_i}\lambda_i,
    \qquad
    i=1,\ldots,p,
    \qquad
    \lambda'_j=0,
    \qquad
    j=p+1,\ldots,d.
  \end{equation*}
  Since $\cP(\pmb\lambda')$ is evidently a~subset of $\mu_p\big(\cP(\pmb\lambda),\La'\big)\cP(\pmb\lambda)$, it is also a~subset of $\cC_\e$.

  Set $\delta$ to be equal to the distance from $\cL$ to the nearest point of $\La\backslash\cL$. Consider the $(d-p)$-dimensional parallelepiped $\cP(\pmb\lambda'')$, where $\pmb\lambda''=(\lambda''_1,\ldots,\lambda''_d)$,
  \[
    \lambda''_i=0,
    \quad
    i=1,\ldots,p,
    \quad
    \lambda''_j=\delta\big/\big(2\sqrt{d-p}\big),
    \qquad
    j=p+1,\ldots,d.
  \]
  Then $\cP(\pmb\lambda'')$ is contained in the $\delta/2$-neighborhood of the origin $\vec 0$. If $\e<\delta\big/\big(4\sqrt{d-p}\big)$, then the Minkowski's sum $\cP(\pmb\lambda')+\cP(\pmb\lambda'')=\cP(\pmb\lambda'+\pmb\lambda'')$ is contained in the open $\delta$-neighborhood of $\ell$, and thus, does not contain any points of $\La\backslash\cL$. On the other hand, $\cP(\pmb\lambda')$ is the orthogonal projection of $\cP(\pmb\lambda'+\pmb\lambda'')$ onto $\cL'$, hence the fact that there are no nonzero points of $\La'$ in $\cP(\pmb\lambda')$ implies that there are no nonzero points of $\La\cap\cL$ in $\cP(\pmb\lambda'+\pmb\lambda'')$.

  Thus, the parallelepiped $\cP(\pmb\lambda'+\pmb\lambda'')$ contains no nonzero points of $\La$, whereas its volume
  \[
    \frac{2^d (c')^p \delta^{d-p}}{\big(2\sqrt{d-p}\big)^{d-p}}
    \cdot
    \prod_{k=1}^{p}\mu_k,
  \]
  due to~\eqref{eq:prod_of_the_mus}, is bounded away from zero by a~constant that depends only on $\La$ and $\cL$. Taking into account that $\e$ is arbitrarily small, we get $\hat\omega(\La)=0$.
\end{proof}

\section{Concerning the ``weak'' version of uniform lattice exponents and multiplicative Diophantine exponents}\label{sec:weak_uniform_exponents}

\subsection{Weak uniform Diophantine exponents}

Changing one of the quantifiers in Definition~\ref{def:uniform_lattice_exponents} leads to the following ``weak'' version of a~uniform Diophantine exponent of a~lattice.

\begin{definition}\label{def:weak_uniform_lattice_exponents}
  The supremum of real numbers $\gamma$ such that for every $t$ large enough there exists a~$d$-tuple $\pmb\lambda\in\R_+^d$ satisfying
  \[
    |\pmb\lambda|=t,
    \qquad
    \Pi(\pmb\lambda)=t^{-\gamma},
  \]
  such that
  the parallelepiped $\cP(\pmb\lambda)$ contains a~nonzero point of $\La$, is called the \emph{weak uniform Diophantine exponent} of $\La$ and is denoted by $\breve\omega(\La)$.
\end{definition}

It is clear that the argument provided in Section~\ref{sec:proof} proves nothing in the case of weak uniform lattice exponents. Thus, the following problem is still unsolved.

\begin{problem}
  Describe the spectrum of the exponent $\breve\omega(\La)$.
\end{problem}

It is also easy to see that $\breve\omega(\La)$ is equal to the supremum of real numbers $\gamma$ such that for every $t$ large enough the system of inequalities
\begin{equation}\label{eq:weak_uniform_lattice_exponents}
  |\vec z|\leq t,
  \qquad
  \Pi(\vec z)\leq t^{-\gamma}
\end{equation}
admits a~nonzero solution $\vec z\in\La$. This resembles a~lot how multiplicative Diophantine exponents are defined.

\subsection{Multiplicative Diophantine exponents}

As we mentioned in the Introduction, the case $n<d$ leads to multiplicative Diophantine exponents. In order to give the respective definitions, along with the functional $\Pi(\cdot)$, we will need the functional $\Pi'(\cdot)$. Given a~positive integer $k$ and a~$k$-tuple $\vec z=(z_1,\ldots,z_k)\in\R^k$, we denote
\[
  \Pi(\vec z)=\prod_{
\begin{subarray}
{c}1\leq i\leq k
\end{subarray}
}|z_i|^{1/k},
  \qquad
  \Pi'(\vec z)=\prod_{1\leq i\leq k}\max\big(1,|z_i|\big)^{1/k}.
\]

Suppose $n<d$. Then, without loss of generality, we may assume that the coefficients of the forms $L_1,\ldots,L_n$ are written in the rows of the matrix
\[
  \begin{pmatrix}
    \Theta & -\vec I_n
  \end{pmatrix}
=
  \begin{pmatrix}
    \theta_{11} & \cdots & \theta_{1m} & -1 & \cdots & 0 \\
    \vdots & \ddots & \vdots & \vdots & \ddots & \vdots \\
    \theta_{n1} & \cdots & \theta_{nm} & 0 & \cdots & -1
  \end{pmatrix}
.
\]

\begin{definition}\label{def:mbeta}
  Supremum of real numbers $\gamma$ for which there exists $t$ however large such that the system of inequalities
  \begin{equation}\label{eq:mbeta}
    \Pi'(\vec x)\leq t,
    \qquad
    \Pi(\Theta\vec x-\vec y)\leq t^{-\gamma}
  \end{equation}
  admits a~solution $(\vec x,\vec y)\in\Z^m \oplus\Z^n$ with nonzero $\vec x$ is called the \emph{multiplicative Diophantine exponent} of $\Theta$ and is denoted by $\omega_\times(\Theta)$.

  The respective \emph{uniform multiplicative Diophantine exponent} $\hat\omega_\times(\Theta)$ is defined as the supremum of real numbers $\gamma$ such that~\eqref{eq:mbeta} admits solutions $(\vec x,\vec y)\in\Z^m \oplus\Z^n$ with nonzero $\vec x$ for every $t$ large enough.
\end{definition}

As in the case of lattice exponents, it is a~straightforward consequence of Minkowski's convex body theorem that
\[
  \omega_\times(\Omega)\geq\hat\omega_\times(\Omega)\geq m/n.
\]
Wang and Yu proved in~\cite{wang_yu} that $\omega_\times(\Omega)=m/n$ holds for almost all $\Theta$ with respect to the Lebesgue measure. Clearly, in this case $\hat\omega_\times(\Omega)$ also equals $m/n$. Apart from this result nothing is known so far about the set of admissible values of $\omega_\times(\Omega)$ or $\hat\omega_\times(\Omega)$. Thus, the following problem is also currently far from its solution.

\begin{problem}
  Describe the spectra of the exponents $\omega_\times(\Omega)$ and $\hat\omega_\times(\Omega)$.
\end{problem}

\subsection*{Acknowledgments.}
The author is a~winner of the ``Junior Leader'' contest conducted by Theoretical Physics and Mathematics Advancement Foundation ``BASIS'' and would like to thank its sponsors and jury.

This research was supported by the Russian Science Foundation (grant 22-41-05001), https://rscf.ru/en/project/22-41-05001/.


\EditInfo{March 30, 2023}{April 28, 2023}{Camilla Hollanti and Lenny Fukshansky}

\end{document}